\documentclass[10pt]{article}
\usepackage[cp1251]{inputenc}
\usepackage[ukrainian,russian]{babel}
\usepackage{latexsym,amsfonts,amssymb,amsmath,amsthm,graphicx,epsfig,epstopdf,cite,dsfont,hyperref}

\newcommand{\UDC}[1]{\vspace{-24pt}\begin{flushleft}\tt #1\end{flushleft}}
\newcommand{\Title}[1]{\begin{flushleft}\Large \bf #1\end{flushleft}}
\newcommand{\Author}[1]{\begin{flushleft}\bfseries\itshape #1\end{flushleft}}
\newcommand{\Address}[1]{\begin{flushleft}\it #1\end{flushleft}}
\newcommand{\Abstract}[1]{\begin{flushright}
\centering\begin{minipage}{148mm}\small #1\end{minipage}\end{flushright}}

\theoremstyle{plain}
\newtheorem{theorem}{\hspace{\parindent}\bf Теорема}
\newtheorem{lemma}{\hspace{\parindent}\bf Лема}

\newtheorem{algorithm}{\hspace{\parindent}\bf Алгоритм}

\theoremstyle{definition}

\newtheorem{remark}{\hspace{\parindent}\bf Зауваження}

\renewenvironment{proof}[1][\textbf{\textit{Доведення}}]{\emph{#1. }}

\begin{document}
\large

\UDC{УДК 517.925.51;\,681.5.03}

\noindent {\large \textbf{О.\,Г. Мазко}} (Ін-т математики НАН України, Київ)

\Title{Еквівалентне перетворення і зважена $H_\infty$-оптимізація лінійних дескрипторних систем}



\vspace{1mm}

\Abstract{The problem of a generalized type of $H_\infty$-control is investigated for a class of admissible descriptor systems with a non-zero initial vector. A generalized performance measure is used, which characterizes the weighted damping level of external and initial disturbances.
A non-degenerate transformation of the system is proposed, which allows to apply known methods for evaluation and the achievement of desired performance measures for conventional lower-order systems. A numerical example of a controlled hydraulic system with three tanks is given.}


\Abstract{Досліджується проблема узагальненого $H_\infty$-керування для класу допустимих дескрипторних систем з ненульовим початковим вектором.
Використовується узагальнений критерій якості, що характеризує зважений рівень гасіння зовнішніх і початкових збурень.
Пропонується невироджене перетворення системи, яке дозволяє застосовувати відомі методи оцінки та досягнення бажаних критеріїв якості для звичайних систем меншого порядку. Наводиться числовий приклад керованої гідравлічної системи з трьома резервуарами.}

\medskip

\textbf{1. Вступ.}
В сучасній теорії керування велика увага приділяється дескрипторним (диференціально-алгебраїчним) системам, які застосовуються при моделюванні руху об'єктів механіки, робототехніки, енергетики, електротехніки, економіки тощо (див., наприклад, \cite{Dai-89,Duan-2010,Feng-2017,Campbell-2019}).
Рівняння руху, входів та виходів керованих об'єктів можуть містити невизначені елементи (параметри, зовнішні збурення, неточності вимірів тощо), які обумовлюють необхідність розв'язання проблем робастної стабілізації та мінімізації впливу обмежених збурень на якість перехідних процесів.

Типовим критерієм якості у задачах $H_\infty$-оптимізації систем з нульовим початковим станом є рівень гасіння зовнішніх збурень, якому відповідає максимальне значення відношення $L_2$-норм векторів керованого виходу об'єкта і збурень. Так, для лінійної дескрипторної системи
\begin{equation}\label{f1-1}
 E\dot{x}=Ax+Bw,\quad z=Cx+Dw,
\end{equation}
дана характеристика збігається з $H_\infty$-нормою матричної передатної функції
$$
\|\mathcal{H}\|_\infty=\sup_{\omega\in {\mathbb R}} \sqrt{\lambda_{\max}(\mathcal{H}^\top (-i\omega)\mathcal{H}(i\omega))},\quad  \mathcal{H}(\lambda)=C(\lambda E-A)^{-1}B+D,
$$
де $x\in {\mathbb R}^n$, $z\in {\mathbb R}^k$ і $w\in {\mathbb R}^s$ --- вектори відповідно стану, контрольованого виходу і входу системи,
$\lambda_{\max}(\cdot)$ --- максимальне власне значення матриці.
На практиці доцільно застосовувати узагальнений критерій якості \cite{Mazko-2016,Mazko-2023}
\begin{equation}\label{f1-2}
J=\sup\limits_{\{w,x_0\}\in \mathcal{W}} \frac{\|z\|_Q}{\sqrt{\|w\|^2_P+x_0^\top X_0x_0}}.
\end{equation}
Тут $\|z\|_Q$ і $\|w\|_P$ --- зважені $L_2$-норми відповідних векторів $z$ і $w$ вигляду
$$ \|z\|_Q=\sqrt{\int_0^\infty z^\top Qz\, dt},\quad \|w\|_P=\sqrt{\int_0^\infty w^\top Pw\, dt},$$
$\mathcal{W}$ --- множина допустимих пар \{$w,x_0$\} системи, для яких виконується нерівність $0<\|w\|^2_P+x_0^\top X_0x_0< \infty$, а $P=P^\top>0$, $Q=Q^\top>0$ і $X_0=E^\top HE$ --- задані вагові матриці, причому $H=H^\top>0$ і початковий вектор $x_0=x(0_-)$ (див. також \cite{Bal-Kog-2010,Feng-2016}). Вираз (\ref{f1-2})  при $x_0\in {\rm Ker}\,E$ позначимо як $J_0$.
Очевидно, що $J_0\le J$. Якщо $P=I_s$ і $Q=I_l$, то $J_0=\|\mathcal{H}\|_\infty$.
Значення $J$ характеризує зважений рівень гасіння зовнішніх збурень, а також початкових збурень, обумовлених ненульовим початковим вектором.

Відомі методи синтезу $H_\infty$-керування використовують верхні оцінки відповідних критеріїв якості, сформульовані у термінах розв'язків квадратичних матричних рівнянь типу Ріккаті та лінійних матричних нерівностей (ЛМН) \cite{Boyd-94,Gahinet-Apkarian-94,Xu-2007}. Для класу лінійних дескрипторних систем такі оцінки встановлено в \cite{Chadli-18,Gao-99,Masubushi-97,Mazko_UMJ-18}. З відомими методами $H_\infty$-оптимізації дескрипторних систем можна ознайомитись, наприклад, в \cite{Chadli-18,Masubushi-97,Feng-2017,Inoue-15,Mazko-2023}.

В даній роботі пропонуються новий підхід до розв'язання узагальненої проблеми $H_\infty$-керування для лінійних дескрипторних систем з критеріями якості типу (\ref{f1-2}) на основі невиродженого перетворення таких систем до звичайних і застосування відомих методів синтезу статичних та динамічних регуляторів. Як наслідок, у ряді випадків відповідні алгоритми синтезу керувань базуються на розв'язанні ЛМН без додаткових рангових обмежень. Зокрема, порядок шуканого динамічного регулятора у таких алгоритмах синтезу не перевищує рангу матриці при похідних вихідної системи.
Також, відмітна особливість отриманих результатів порівняно з відомими полягає у застосуванні зважених критеріїв якості, які дають нові можливості при досягненні бажаних характеристик дескрипторних систем керування. За допомогою вагових коефіцієнтів у таких критеріях якості можна встановити пріоритети між компонентами керованого виходу і невизначених збурень у системі керування. Причому, компонентами невизначених збурень можуть бути як зовнішні збурення, що діють на систему, так і похибки вимірюваного виходу.

Слід зазначити, що для розв'язання ЛМН створено достатньо ефективні комп'ютерні засоби, наприклад, LMI Toolbox системи MATLAB \cite{Gahinet-95}. Для розв'язання ЛМН з ранговими обмеженнями можна застосувати засоби LMIRank i YALMIP системи MATLAB \cite{Orsi-Helmke-Moore-2006,Lofberg-2004} або блок розв'язку (Solve Block) системи Mathcad Prime \cite{Vosk-Zad_2016}.

Будемо використовувати такі позначення:
$I_n$ --- одинична матриця порядку $n$;
$0_{n\times m}$ --- нульова матриця розмірів $n\times m$;
$X=X^\top>0$ ($\ge 0$) --- додатно (невід'ємно) визначена матриця $X$;
$\sigma(A)$ --- спектр матриці $A$;
$A^{-1}$($A^+$) --- обернена (псевдообернена) матриця;
${\rm Ker}\, A$ --- ядро матриці $A$;
$W_A$ --- матриця, стовпці якої утворюють базис ядра ${\rm Ker}\, A$;
$\|w\|_P$ --- зважена $L_2$-норма вектор-функції $w(t)$; ${\mathbb C}^-$ --- відкрита півплощина ${\rm Re \lambda}<0$.

\medskip

\textbf{2. Означення і допоміжні твердження.}
Розглянемо дескрипторну систему (\ref{f1-1}), де ${\rm rank}\,E=r<n$, і її критерій якості (\ref{f1-2}).  Система (\ref{f1-1}) називається \emph{допустимою}, якщо пара матриць \{$E,A$\} \emph{регулярна}, \emph{стійка} і \emph{неімпульсна} \cite{Dai-89}, тобто відповідно $\det F(\lambda)\not\equiv 0$ ($\lambda \in \mathbb{C}$), $\sigma(F)\subset {\mathbb C}^-$ і ${\rm deg}\,\{\det F(\lambda)\}=r $, де $\sigma(F)=\{\lambda_1,\dots,\lambda_\alpha\}$ --- скінченний спектр в'язки матриць $F(\lambda)=A-\lambda E$.
Система (\ref{f1-1}) називається \emph{внутрiшньо стiйкою}, якщо вона стійка у вiдсутностi збурень ($w \equiv 0$).

Пара матриць \{$E,A$\} є регулярною тоді і лише тоді, коли існують невироджені матриці $L$ і $R$, що перетворюють її до канонічної форми Веєрштрасса \cite{Gantmaher}.
Система (\ref{f1-1}) є неімпульсною тоді і лише тоді, коли \cite{Duan-2010}
\begin{equation}\label{f2-1}
{\rm rank}\,\left[\begin{array}{cc}
                  E & 0 \\
                  A & E \\
                \end{array}\right]=n+r.
\end{equation}

Нехай $E=E_1 E_2^\top$ --- скелетна декомпозиція матриці $E$, де $E_1,E_2\in \mathbb{R}^{n\times r}$ --- матриці повного рангу $r$, а $E_1^\bot,E_2^\bot\in \mathbb{R}^{n\times (n-r)}$ --- відповідні ортогональні доповнення, що задовольняють співвідношення
$$E_i^\top E_i^\bot=0,\quad \det\,\left[\begin{array}{cc}
                                   E_i & E_i^\bot \\
                            \end{array}\right]\ne 0,\quad
i=1,2.$$

Виконаємо невироджене перетворення системи (\ref{f1-1}) на основі співвідношень
\begin{equation}\label{f2-2}
 LER=\left[\begin{array}{cc}
          I_r & 0 \\
          0 & 0 \\
        \end{array}\right],\;
  LAR=\left[\begin{array}{cc}
          A_1 & A_2 \\
          A_3 & A_4 \\
        \end{array}\right],\;
        x=R\left[\begin{array}{c}
          \xi_1 \\
          \xi_2 \\
        \end{array}\right],\;  \xi_1\in \mathbb{R}^r, \;  \xi_2\in \mathbb{R}^{n-r},
 \end{equation}
 де
 $$L=\left[\begin{array}{c}
            E_1^+ \\
          E_1^{\bot +} \\
        \end{array}\right],\quad  E_1^+=(E_1^\top E_1)^{-1}E_1^\top,\quad E_1^{\bot+}=(E_1^{\bot\top} E_1^{\bot})^{-1}E_1^{\bot\top}, $$
 $$R=\left[\begin{array}{cc}
       E_2^{+\top} & E_2^{\bot +\top} \\
     \end{array}\right],\quad E_2^+=(E_2^\top E_2)^{-1}E_2^\top,\quad
                E_2^{\bot+}=(E_2^{\bot\top} E_2^{\bot})^{-1}E_2^{\bot\top},$$
$$
A_1=E_1^+AE_2^{+\top},\quad A_2=E_1^+AE_2^{\bot+\top},\quad A_3=E_1^{\bot+}AE_2^{+\top},\quad A_4=E_1^{\bot+}AE_2^{\bot+\top}.
$$
При цьому
$$L^{-1}=\left[\begin{array}{cc}
             E_1 & E_1^\bot\\
          \end{array}\right],\quad
R^{-1}=\left[\begin{array}{c}
 E_2^\top \\
 E_2^{\bot\top} \\
 \end{array}\right],\quad \xi_1=E_2^\top x,\quad \xi_2=E_2^{\bot\top} x.$$

Неважко встановити, що умова (\ref{f2-1}) еквівалентна нерівності $\det\,A_4\ne 0$, тобто
\begin{equation}\label{f2-3}
 \det\,(E_1^{\bot \top}A E_2^{\bot})\ne 0.
\end{equation}
Виключаючи змінну $\xi_2= -A_4^{-1}\big(A_3\xi_1+B_2w\big)$ за умови (\ref{f2-3}), на основі перетворення (\ref{f2-2}) отримаємо звичайну систему
\begin{equation}\label{f2-4}
  \dot{\xi}_1=\bar{A}\xi_1+\bar{B} w,\quad z=\bar{C} \xi_1+\bar{D} w,\quad \xi_1(0)=\xi_{10},
\end{equation}
де
$$\bar{A}=A_1-A_2A_4^{-1}A_3,\; \bar{B}=B_1-A_2A_4^{-1}B_2,\; \bar{C}=C_1-C_2A_4^{-1}A_3,\; \bar{D}=D-C_2A_4^{-1}B_2,$$
$$  LB=\left[\begin{array}{c}
          B_1 \\
          B_2 \\
        \end{array}\right],\quad CR= \left[\begin{array}{cc}
                                         C_1 & C_2 \\
                                       \end{array}\right].$$
Зазначимо, що спектр матриці $\bar{A}$ збігається з $\sigma(F)$, а критерії якості $J_0$ і $J$ неімпульсної системи (\ref{f1-1}) не залежать від $\xi_2$ і визначається системою (\ref{f2-4}), оскільки
$$
\left[\begin{array}{cc}
    I_r & -A_2A_4^{-1} \\
    0 & I_{n-r} \\
  \end{array}\right]LF(\lambda)R
  \left[\begin{array}{cc}
    I_r & 0 \\
    -A_4^{-1}A_3 & I_{n-r} \\
  \end{array}\right]=\left[\begin{array}{cc}
                         \bar{A}-\lambda I_r & 0 \\
                         0 & A_4 \\
                       \end{array}\right],
$$
$$
x_0^\top X_0x_0=\left[\begin{array}{cc}
            \xi_{10}^\top & \xi_{20}^\top\\
         \end{array}\right]R^\top E^\top L^\top L^{-1\top}HL^{-1}LER\left[\begin{array}{c}
                                                                                                  \xi_{10} \\
                                                                                                  \xi_{20} \\
                                                                                                \end{array}\right]=\xi_{10}^\top \bar{H}\xi_{10},$$
де $\bar{H} =E_1^\top H E_1$.

Отже, застосовуючи \cite[лема 4.1]{Mazko-NDST-2017} до системи (\ref{f2-4}), маємо таке твердження.

\begin{lemma}\label{l2-1}
Система {\rm (\ref{f1-1})} є допустимою і $J_0<\gamma$ тоді і лише тоді, коли виконується умова {\rm (\ref{f2-3})} і для деякої матриці $X=X^\top>0$
\begin{equation}\label{f2-5}
{\bar{\Omega}(X)}=\left[\begin{array}{ccc}
                   \bar{A}^\top X+X\bar{A} & X \bar{B} & \bar{C}^\top\\
                   \bar{B}^\top X & -\gamma^2P & \bar{D}^\top \\
                   \bar{C} & \bar{D} & - Q^{-1}\\
                 \end{array}\right]< 0.
\end{equation}
Система {\rm (\ref{f1-1})} допустима і $J<\gamma$ тоді і лише тоді, коли виконується умова {\rm (\ref{f2-3})} і система ЛМН {\rm (\ref{f2-5})} і
\begin{equation}\label{f2-6}
0<X<\gamma^2 \bar{H}
\end{equation}
сумісна щодо $X$.
\end{lemma}

Із леми \ref{l2-1} випливають алгоритми обчислення характеристик $J_0$ і $J$ системи (\ref{f1-1}) на основі розв'язування відповідних оптимізаційних задач з обмеженнями виключно в термінах ЛМН:
$$J_0=\inf \big\{\gamma:\; \bar{\Omega}(X)<0,\; X>0\big\},\quad
 J=\inf \big\{\gamma:\; \bar{\Omega}(X)<0,\; 0<X<\gamma^2\bar{H}\big\}.$$

\medskip
Вектор збурення $w(t)$ і початковий вектор $x_0$ називатимемо \emph{найгіршими} в системі (\ref{f1-1}) щодо критерію якості $J$, якщо на їхніх значеннях в (\ref{f1-2}) досягається супремум, тобто $\|z\|_Q^2=J^2\big(\|w\|^2_P+x_0^\top X_0x_0\big)$. Методи знаходження таких векторів у окремих випадках запропоновано в \cite{Bal-Kog-2010,Mazko-Kusii-18,Mazko-Kotov-UMJ-19}.
Так, якщо система {\rm(\ref{f1-1})} допустима і для деякої матриці $X$ виконуються співвідношення
$$
  A_0^\top X+X^\top A_0+X^\top R_0X+Q_0=0,\quad 0\le E^\top X=X^\top E\le \gamma^2 X_0,
$$
де \;$A_0=A+BR_1^{-1}D^\top QC$, \;$R_0=BR_1^{-1}B^\top $, \;$Q_0=C^\top \!\big(Q+QDR_1^{-1}D^\top Q\big)C$, $R_1=\gamma^2P-D^\top QD>0$ і $\gamma=J$,
то структурований вектор зовнішніх збурень у формі лінійного зворотного зв'язку за станом
\begin{equation*}\label{f2-7}
  w=K_*x,\quad K_*=R_1^{-1}(B^\top X+D^\top QC),
\end{equation*}
і довільний початковий вектор $x_0\in {\rm Ker}\,(E^\top X-J^2 X_0)$ є найгіршими щодо критерію якості $J$ для системи {\rm(\ref{f1-1})} \cite{Mazko-Kotov-UMJ-19}.

Наведемо інший спосіб знаходження найгіршої пари \{$w(t),x_0$\} неімпульної системи (\ref{f1-1}) щодо критерію якості $J$ із застосуванням перетворення (\ref{f2-2}). За умови (\ref{f2-3}) найгірший початковий вектор будуємо у вигляді
\begin{equation}\label{f2-8}
x_0=R\left[\begin{array}{c}
          \xi_{10} \\
          -A_4^{-1}\big(A_3\xi_{10}+B_2w(0)\big) \\
        \end{array}\right],
\end{equation}
де \{$w(t),\xi_{10}$\} --- найгірша пара системи (\ref{f2-4}) щодо критерію якості $J$.

Згідно з лемою Шура умова (\ref{f2-5}) еквівалентна матричній нерівності Ріккаті
\begin{equation}\label{f2-9}
\bar{A}_0^\top X+X\bar{A}_0+X\bar{R}_0X+\bar{Q}_0<0,
\end{equation}
де $\bar{A}_0=\bar{A}+\bar{B}\,\bar{R}_1^{-1}\bar{D}^\top Q\bar{C}$, $\bar{Q}_0=\bar{C}^\top \big(Q+Q\bar{D}\,\bar{R}_1^{-1}\bar{D}^\top Q\big)\bar{C}$, $\bar{R}_0=\bar{B}\,\bar{R}_1^{-1}\bar{B}^\top $ і $\bar{R}_1=\gamma^2P-\bar{D}^\top Q\bar{D}>0$.
Якщо пара матриць \{$\bar{A},\bar{B}$\} керована, пара матриць \{$\bar{A},\bar{C}$\} спостережувана і $J_0<\gamma$, то відповідне матричне рівняння Ріккаті
\begin{equation}\label{f2-10}
\bar{A}_0^\top X+X\bar{A}_0+X\bar{R}_0X+\bar{Q}_0=0
\end{equation}
має розв'язки $X_-$ і $X_+$ такі, що $\sigma(\bar{A}_0+\bar{R}_0 X_\pm)\subset \mathbb{C}^\pm$, $0<X_-<X_+$ і кожний розв'язок нерівності (\ref{f2-9}) належить інтервалу $X_-< X< X_+$ (див. \cite{Zhou-Doyle-Glover-1996,Dullerud-Paganini-2000}). Більш того, якщо $J<\gamma$ ($J\le\gamma$) і $X$ задовольняє рівняння (\ref{f2-10}), то $X<\gamma^2\bar{H}$ ($X\le\gamma^2\bar{H}$). Справді, поклавши $v(\xi_1)=\xi_1^\top X\xi_1$ і
\begin{equation}\label{f2-11}
w=\bar{K}_*\xi_1,\quad \bar{K}_*=\bar{R}_1^{-1}\big(\bar{B}^\top X+\bar{D}^\top Q\bar{C}\big),
\end{equation}
отримаємо рівність $\dot{v}+z^\top Qz-\gamma^2w^\top Pw=0,$
де $\dot{v}$ --- похідна функції $v(\xi_1)$ в силу системи (\ref{f2-4}). Після інтегрування даної рівності на нескінченному інтервалі за умови $J<\gamma$ отримаємо $\|z\|_Q^2-\gamma^2 \|w\|_P^2=\xi_{10}^\top X\xi_{10}< \gamma^2\xi_{10}^ \top \bar{H}\xi_{10}$
для будь-якого $\xi_{10}\ne 0$, інакше $J\ge \gamma$. Якщо ж $J=\gamma$, то за умов (\ref{f2-10}) і (\ref{f2-11}) рівність
$\xi_{10}^\top X\xi_{10}= \gamma^2\xi_{10}^\top \bar{H} \xi_{10}$ або їй еквівалентна $(X-\gamma^2\bar{H})\xi_{10}=0$ можлива при ненульовому значенні $\xi_{10}$. При цьому $\|z\|_Q^2=J^2(\|w\|_P^2+\xi_{10}^\top \bar{H}\xi_{10})$, тобто в (\ref{f1-2}) досягається супремум.
У випадку $x_0\in {\rm Ker}\,E$ за умов (\ref{f2-10}) і (\ref{f2-11}) маємо $\xi_{10}=0$ і $\|z\|_Q=J_0\|w\|_P$, тобто (\ref{f2-11}) є найгіршим збуренням щодо критерію якості $J_0$.

Отже, виконується таке твердження.

\begin{lemma}\label{l2-2}
Нехай $X>0$ --- стабілізуючий розв'язок рівняння Ріккаті {\rm(\ref{f2-10})} за умов {\rm(\ref{f2-3})} і $\gamma=J$. Тоді структурований вектор зовнішніх збурень {\rm(\ref{f2-11})}, де $\xi_1=\xi_1(t,\xi_{10})$ --- розв'язок системи
\begin{equation}\label{f2-12}
\dot{\xi}_1=(\bar{A}+\bar{B}\bar{K}_*)\xi_1,\quad \xi_1(0)=\xi_{10},
\end{equation}
і довільний вектор {\rm(\ref{f2-8})} при $\xi_{10}\in {\rm Ker}\,(X-J^2 \bar{H})$ є найгіршими щодо критерію якості $J$ для системи {\rm(\ref{f1-1})}.

Якщо $X>0$ --- стабілізуючий розв'язок рівняння {\rm(\ref{f2-10})} за умов {\rm(\ref{f2-3})} і $\gamma=J_0$, а $\xi_1=\xi_1(t,\xi_{10})$ --- розв'язок системи {\rm(\ref{f2-12})} при $\xi_{10}=0$, то {\rm(\ref{f2-11})} є найгіршим збуренням щодо критерію якості $J_0$ для системи {\rm(\ref{f1-1})}.
\end{lemma}

\medskip
\textbf{3. Основні результати.}
Розглянемо дескрипторну систему керування
\begin{equation}\label{f3-1}
   E\dot{x}=Ax+B_1w+B_2u,\quad
    z=C_1x+D_{11}w+D_{12}u,\quad
    y=C_2x+D_{21}w+D_{22}u,
\end{equation}
де $x\in {\mathbb R}^n$, $u\in {\mathbb R}^{m}$, $w\in {\mathbb R}^{s}$, $z\in {\mathbb R}^{k}$ і $y\in {\mathbb R}^{l}$ --- вектори відповідно стану, керування, зовнішніх збурень, контрольованого і спостережуваного виходів. Всі матричні коефіцієнти в (\ref{f3-1}) сталі, причому пара \{$E,A$\} регулярна, неімпульсна і ${\rm rank}\,E=r< n$. Компонентами вектора $w(t)$ можуть бути як зовнішні збурення, що діють на систему, так i похибки вимiрюваного виходу. Даний вектор має бути обмеженим за зваженою $L_2$-нормою. Початкові збурення в системі обумовлені невідомим початковим вектором $x_0=x(0_-)$.

Нас цікавлять стабілізуючі закони керування, які гарантують внутрішню стійкість замкненої системи та бажані верхні оцінки її критеріїв якості (\ref{f1-2}) стосовно контрольованого виходу $z$. Статичні й динамічні регулятори, які мінімізують критерію якості $J$, називатимемо $J$-\emph{оптимальними}. Пошук $J_0$- та $J$-оптимальних регуляторів можна здійснювати на основі досягнення відповідних оцінок $J_0<\gamma$ та $J<\gamma$ при мінімально можливому значенні $\gamma$.

При дослідженні класу систем (\ref{f3-1}) використовуються такі їхні властивості, як $C$-, $R$- та $I$-\emph{керованість}, а також двоїсті до них $C$-, $R$- та $I$-\emph{спостережуваність} \cite{Feng-2017}. Зокрема, для розв'язання узагальнених задач $H_\infty$-оптимізації необхідно, щоб трійка матриць \{$E,A,B_2$\} була стабілізовною та $I$-керованою. Це означає, що повинна існувати така матриця $K$, щоб пара матриць \{$E,A+B_2K$\} була стійкою і неімпульсною, тобто допустимою. Критеріями $I$-керованості трійки \{$E,A,B_2$\} та $I$-спостережуваності трійки \{$E,A,C_2$\} є відповідні рівності \cite{Cobb-84}
\begin{equation}\label{f3-1-1}
{\rm rank}\,\left[\begin{array}{ccc}
                E & 0 & 0 \\
                A & E & B_2 \\
              \end{array}\right]=n+r,\quad
{\rm rank}\,\left[\begin{array}{cc}
                E & A \\
                0 & E \\
                0 & C_2 \\
              \end{array}\right]=n+r.
\end{equation}

Застосуємо до системи (\ref{f3-1}) еквівалентне перетворення (\ref{f2-2}). Виключаючи змінну $\xi_2= -A_4^{-1}\big(A_3\xi_1+B_{12}w+B_{22}u\big)$ за умови (\ref{f2-3}), отримаємо звичайну систему
 \begin{equation}\label{f3-2}
    \dot{\xi}_1=  \bar{A}\xi_1+ \bar{B}_1w+\bar{B}_2u,\quad
    z=\bar{C}_1\xi_1+\bar{D}_{11}w+\bar{D}_{12}u, \quad
    y=\bar{C}_2\xi_1+\bar{D}_{21}w+\bar{D}_{22}u,
\end{equation}
де
$$\bar{A}=A_1-A_2A_4^{-1}A_3,\quad \bar{B}_1=B_{11}-A_2A_4^{-1}B_{12},\quad \bar{B}_2=B_{21}-A_2A_4^{-1}B_{22},$$
$$\bar{C}_1=C_{11}-C_{12}A_4^{-1}A_3,\quad \bar{D}_{11}=D_{11}-C_{12}A_4^{-1}B_{12},\quad \bar{D}_{12}=D_{12}-C_{12}A_4^{-1}B_{22},$$
$$\bar{C}_2=C_{21}-C_{22}A_4^{-1}A_3,\quad \bar{D}_{21}=D_{21}-C_{22}A_4^{-1}B_{12},\quad \bar{D}_{22}=D_{22}-C_{22}A_4^{-1}B_{22},$$
$$ LB_1=\left[\begin{array}{c}
          B_{11} \\
          B_{12} \\
        \end{array}\right],\quad LB_2=\left[\begin{array}{c}
          B_{21} \\
          B_{22} \\
        \end{array}\right],\quad C_1R=\left[\begin{array}{cc}
                                         C_{11} & C_{12} \\
                                       \end{array}\right],\quad C_2R=\left[\begin{array}{cc}
                                         C_{21} & C_{22} \\
                                       \end{array}\right].$$
В означенні критерію якості (\ref{f1-2}) даної системи використовуємо вираз $x_0^\top X_0x_0=\xi_{10}^\top \bar{H}\xi_{10}$, де $\xi_{10}=\xi_1(0)$, $\bar{H}=E_1^\top H E_1$ (див. попередній розділ).

Таким чином, задачі $J_0$- та $J$-оптимізації дескрипторної системи (\ref{f3-1}) з неімпульсною парою \{$E,A$\} зводяться до розв'язання аналогічних задач для звичайної системи (\ref{f3-2}).

\medskip
\textbf{3.1. Статичний регулятор.} Для системи (\ref{f3-2}) побудуємо статичний регулятор за виходом
\begin{equation}\label{f3-3}
 u=Ky,\quad K\in \mathbb{R}^{m\times l}.
\end{equation}
Замкнена система за умови $\det(I_m-K\bar{D}_{22})\ne 0$ має вигляд
\begin{equation}\label{f3-3-1}
 \dot{\xi}_1=A_*\xi_1+B_*w,\quad z=C_*\xi_1+D_*w,
\end{equation}
де $A_*=\bar{A}+\bar{B}_2K_0\bar{C}_2$, $B_*=\bar{B}_1+\bar{B}_2K_0\bar{D}_{21}$, $C_*=\bar{C}_1+\bar{D}_{12}K_0\bar{C}_2$, $D_*=\bar{D}_{11}+\bar{D}_{12}K_0\bar{D}_{21}$ і $K_0=(I_m-K\bar{D}_{22})^{-1}K$. Керування (\ref{f3-3}) будемо застосовувати також для вихідної системи (\ref{f3-1}).

Матричну нерівність (\ref{f2-5}) у лемі \ref{l2-1} для системи (\ref{f3-3-1}) перепишемо як ЛМН щодо $K_0$:
\begin{equation}\label{f3-7}
  L_0^\top K_0R_0+R_0^\top K_0^\top L_0+\Omega<0,
\end{equation}
де
$$R_0=\left[\begin{array}{ccc}
          \bar{C}_2 & \bar{D}_{21} & 0_{l\times k} \\
        \end{array}\right],\quad L_0=\left[\begin{array}{ccc}
          \bar{B}^\top_2X & 0_{m\times s} & \bar{D}^\top_{12} \\
        \end{array}\right],$$$$
\Omega=\left[\!\begin{array}{ccc}
         \bar{A}^\top X+X\bar{A} & X\bar{B}_1 & \bar{C}_1^\top  \\
         \bar{B}_1^\top X & \!\!-\gamma^2P& \bar{D}_{11}^\top  \\
         \bar{C}_1 & \bar{D}_{11} & \!\!-Q^{-1} \\
         \end{array}\!\right]\!.$$

На основі леми \ref{l2-1} і теореми 5.1 із \cite{Mazko-2023} маємо наступний результат.

 \begin{theorem}\label{t3-1}
Для системи {\rm (\ref{f3-1})} існує статичний регулятор {\rm (\ref{f3-3})}, при якому замкнена система є допустимою і її критерій якості $J<\gamma$, тоді і лише тоді, коли сумісна стосовно $X$ і $Y$ система співвідношень {\rm(\ref{f2-6})} і
\begin{equation}\label{f3-4}
   W_{\bar{R}}^\top \left[\begin{array}{cc}
                 \bar{A}^\top X+X\bar{A}+\bar{C}_1^\top Q\bar{C}_1 & X\bar{B}_1+\bar{C}_1^\top Q\bar{D}_{11} \\
                 \bar{B}_1^\top X+\bar{D}_{11}^\top Q\bar{C}_1 & \bar{D}_{11}^\top Q\bar{D}_{11}-\gamma^2P \\
               \end{array} \right]W_{\bar{R}}<0,
  \end{equation}
\begin{equation}\label{f3-5}
  W_{\bar{L}}^\top \left[\begin{array}{cc}
                 \bar{A}Y+Y\bar{A}^\top +\bar{B}_1P^{-1}\bar{B}_1^\top  & Y\bar{C}_1^\top +\bar{B}_1P^{-1}\bar{D}_{11}^\top  \\
                 \bar{C}_1Y+\bar{D}_{11}P^{-1}\bar{B}_1^\top  & \bar{D}_{11}P^{-1}\bar{D}_{11}^\top -\gamma^2Q^{-1} \\
               \end{array} \right]W_{\bar{L}}<0,
\end{equation}
\begin{equation}\label{f3-6}
  W=\left[\begin{array}{cc}
        X & \gamma I_r \\
        \gamma I_r & Y
      \end{array}
\right]\ge 0,\quad {\rm rank}\,W=r,
\end{equation}
де $\bar{R}=\left[\begin{array}{cc}
                \bar{C}_2 & \bar{D}_{21} \\
              \end{array}\right]$, $\bar{L}=\left[\begin{array}{cc}
                \bar{B}_2^\top & \bar{D}_{12}^\top \\
              \end{array}\right]$.
Матрицю такого регулятора можна знайти у вигляді $K=K_0(I_l+\bar{D}_{22}K_0)^{-1}$,
де $K_0$ --- розв'язок ЛМН {\rm (\ref{f3-7})}.
 \end{theorem}

 Зазначимо, що співвідношення (\ref{f3-6}) виконуються тоді і лише тоді, коли
 \begin{equation}\label{f3-6-2}
   X=X^\top>0,\quad Y=Y^\top>0,\quad XY=\gamma^2 I_r.
 \end{equation}

 Наведемо наслідки леми \ref{l2-1} і теореми \ref{t3-1} за додаткових умов
\begin{equation}\label{f3-6-3}
{\rm rank}\,\bar{C}_2=r\le l,\quad \bar{D}_{21}=0,\quad \bar{D}_{22}=0,
\end{equation}
\begin{equation}\label{f3-6-4}
\bar{D}_{11}^\top Q\bar{D}_{11}<\gamma^2P.
\end{equation}
Умови (\ref{f3-6-3}) виконуються, якщо, наприклад,
$${\rm rank}\,(C_2E_2)=r,\quad C_2E_2^\bot=0,\quad D_{21}=0,\quad D_{22}=0.$$

\begin{theorem}\label{t3-2}
 За умов {\rm (\ref{f3-6-3})} і {\rm (\ref{f3-6-4})} наступні твердження еквівалентні:

{\rm 1)} для системи {\rm (\ref{f3-1})} існує статичний регулятор {\rm (\ref{f3-3})},
при якому замкнена система є допустимою і її критерій якості  $J<\gamma $;

{\rm 2)} існує матриця $Y>\bar{H}^{-1}$, що задовольняє ЛМН {\rm (\ref{f3-5})};

{\rm 3)} існують матриці $Y>\bar{H}^{-1}$ і $Z$, що задовольняють ЛМН
\begin{equation}\label{f3-6-5}
  \left[\begin{array}{ccc}
      \gamma^2(\bar{A}Y+Y\bar{A}^\top\! +\bar{B}_2Z+Z^\top\! \bar{B}_2^\top) & \gamma^2\bar{B}_1 & Y\bar{C}_1^\top+Z^\top\! \bar{D}_{12}^\top \\
      \gamma^2\bar{B}_1^\top & -\gamma^2 P & \bar{D}_{11}^\top \\
      \bar{C}_1Y+\bar{D}_{12}Z & \bar{D}_{11} & -Q^{-1} \\
    \end{array}\right]<0.
\end{equation}

Якщо виконується твердження {\rm 2}, то матрицю $K=K_0$ регулятора {\rm (\ref{f3-3})} у твердженні {\rm 1}  можна знайти як розв'язок ЛМН {\rm (\ref{f3-7})} при $X=\gamma^2 Y^{-1}$. При виконанні твердження {\rm 3} матрицею такого регулятора може бути довільний розв'язок лінійного рівняння  $K\bar{C}_2Y=Z$.
 \end{theorem}
\begin{proof}
Враховуючи умови (\ref{f3-6-3}),
маємо $y=\bar{C}_2\xi_1=\bar{C}_2E_2^\top x$ і $l\ge r$.
Еквівалентність тверджень 1 і 2 випливає з теореми \ref{t3-1}, оскільки за умов (\ref{f3-6-3})
$W_{\bar{R}}=\left[\begin{array}{cc}
         0 & I_s \\
       \end{array}\right]^\top $. При цьому нерівність (\ref{f3-4}) не залежить від $X$ і має вигляд (\ref{f3-6-4}). Шукана матриця в (\ref{f3-5}) згідно з (\ref{f3-6-2}) має вигляд $Y=\gamma^2X^{-1}$. Тому замість (\ref{f2-6}) маємо еквівалентну умову $Y>\bar{H}^{-1}$.
Оскільки $\bar{D}_{22}=0$, то матрицею $K$ регулятора (\ref{f3-3}), що задовольняє твердження 1, може бути довільний розв'язок $K_0$ ЛМН (\ref{f3-7}).

Еквівалентність тверджень 1 і 3 є наслідком леми \ref{l2-1} для замкненої системи (\ref{f3-3-1}), де $K_0=K$. Матрична нерівність (\ref{f3-6-5}) у твердженні 3 виникає в результаті множення першої блочної стрічки зліва і першого блочного стовпця справа виразу (\ref{f3-7}) на $Y=\gamma^2X^{-1}$ з урахуванням умов (\ref{f3-6-3}) і позначення добутку $Z=K\bar{C}_2Y$. Останнє співвідношення за умов (\ref{f3-6-3}) можна розв'язати стосовно $K$:
$$
K=\left\{\begin{array}{ll}
             Z(\bar{C}_2Y)^{-1}, & l=r, \\
             ZY^{-1}\bar{C}_2^++T\bar{C}_2^{\bot \top}, & l>r,
           \end{array}
   \right.
$$
де $T\in \mathbb{R}^{m\times (l-r)}$ --- довільна матриця.

Теорему доведено.
\end{proof}

\medskip

Розглянемо випадок, коли пара матриць \{$E,A$\} у системі (\ref{f3-1}) є імпульною, але існує така матриця $K_1\in \mathbb{R}^{m\times l}$, що
\begin{equation}\label{f3-6-6}
\det\,(I_m-K_1D_{22})\ne 0,\quad \det\big[E_1^{\bot \top}(A+B_2K_{10}C_2)E_2^{\bot}\big]\ne 0,
\end{equation}
де $K_{10}=K_{11}K_1$ і $K_{11}=(I_m-K_1D_{22})^{-1}$.
Можна встановити, що за умов (\ref{f3-6-6}) виконуються рангові співвідношення (\ref{f3-1-1}), тобто система
(\ref{f3-1}) є $I$-керованою та $I$-спостережуваною.

За наведених припущень замість (\ref{f3-3}) застосовуємо регулятор $u=K_1y+v$, де $v$ --- нове керування системи
\begin{equation}\label{f3-6-7}
    E\dot{x}=\widetilde{A}x+\widetilde{B}_1w+\widetilde{B}_2v,\quad
    z=\widetilde{C}_1x+\widetilde{D}_{11}w+\widetilde{D}_{12}v, \quad
    y=\widetilde{C}_2x+\widetilde{D}_{21}w+\widetilde{D}_{22}v.
\end{equation}
Тут за умови (\ref{f3-6-6}) пара матриць \{$E,\widetilde{A}$\} є неімпульсною,
$$\widetilde{A}=A+B_2K_{10}C_2,\quad \widetilde{B}_1=B_1+B_2K_{10}D_{21}, \quad \widetilde{B}_2=B_2K_{11},$$
$$\widetilde{C}_1=C_1+D_{12}K_{10}C_2,\quad \widetilde{D}_{11}=D_{11}+D_{12}K_{10}D_{21},\quad \widetilde{D}_{12}=D_{12}K_{11},$$
$$\widetilde{C}_2=C_2+D_{22}K_{10}C_2,\quad \widetilde{D}_{21}=D_{21}+D_{22}K_{10}D_{21},\quad \widetilde{D}_{22}=D_{22}K_{11}.$$
 Виконавши еквівалентне перетворення системи (\ref{f3-6-7}) на основі співвідношень
$$LER=\left[\begin{array}{cc}
          I_r & 0 \\
          0 & 0 \\
        \end{array}\right],\quad
  L\widetilde{A}R=\left[\begin{array}{cc}
          \widetilde{A}_1 & \widetilde{A}_2 \\
          \widetilde{A}_3 & \widetilde{A}_4 \\
        \end{array}\right],\quad L\widetilde{B}_1=\left[\begin{array}{c}
          \widetilde{B}_{11} \\
          \widetilde{B}_{12} \\
        \end{array}\right],\quad L \widetilde{B}_2=\left[\begin{array}{c}
          \widetilde{B}_{21} \\
         \widetilde{B}_{22} \\
        \end{array}\right],$$
  $$\widetilde{C}_1R=\left[\begin{array}{cc}
                                         \widetilde{C}_{11} & \widetilde{C}_{12} \\
                                       \end{array}\right],\quad \widetilde{C}_2R=\left[\begin{array}{cc}
                                         \widetilde{C}_{21} & \widetilde{C}_{22} \\
                                       \end{array}\right],$$
        $$
        x=R\left[\begin{array}{c}
          \xi_1 \\
          \xi_2 \\
        \end{array}\right],\quad \xi_1=E_2^\top x,\quad  \xi_2= -\widetilde{A}_4^{-1}\big(\widetilde{A}_3\xi_1+\widetilde{B}_{12}w+\widetilde{B}_{22}v\big),$$
де $L$ і $R$ --- невироджені матриці, визначені в (\ref{f2-2}), можна сформулювати аналоги теорем \ref{t3-1} і \ref{t3-2} із застосуванням статичного регулятора $v=Ky$ для звичайної системи
 \begin{equation}\label{f3-2-1}
    \dot{\xi}_1=  \bar{A}\xi_1+ \bar{B}_1w+\bar{B}_2v,\quad
    z=\bar{C}_1\xi_1+\bar{D}_{11}w+\bar{D}_{12}v, \quad
    y=\bar{C}_2\xi_1+\bar{D}_{21}w+\bar{D}_{22}v,
\end{equation}
де
$$\bar{A}=\widetilde{A}_1-\widetilde{A}_2{\widetilde{A}_4}^{-1}\widetilde{A}_3,\quad \bar{B}_1=\widetilde{B}_{11}-\widetilde{A}_2\widetilde{A}_4^{-1}\widetilde{B}_{12},\quad \bar{B}_2=\widetilde{B}_{21}-\widetilde{A}_2\widetilde{A}_4^{-1}\widetilde{B}_{22},$$
$$\bar{C}_1=\widetilde{C}_{11}-\widetilde{C}_{12}\widetilde{A}_4^{-1}\widetilde{A}_3,\quad \bar{D}_{11}=\widetilde{D}_{11}-\widetilde{C}_{12}\widetilde{A}_4^{-1}\widetilde{B}_{12},\quad \bar{D}_{12}=\widetilde{D}_{12}-\widetilde{C}_{12}\widetilde{A}_4^{-1}\widetilde{B}_{22},$$
$$\bar{C}_2=\widetilde{C}_{21}-\widetilde{C}_{22}\widetilde{A}_4^{-1}\widetilde{A}_3,\quad \bar{D}_{21}=\widetilde{D}_{21}-\widetilde{C}_{22}\widetilde{A}_4^{-1}\widetilde{B}_{12},\quad \bar{D}_{22}=\widetilde{D}_{22}-\widetilde{C}_{22}\widetilde{A}_4^{-1}\widetilde{B}_{22}.$$

 В результаті вихідна система (\ref{f3-1}) з керуванням
$$u=\big(K_{10}C_2+K_{11}K_0\bar{C}_2E_2^\top\big)x+\big(K_{10}D_{21}+K_{11}K_0\bar{D}_{21}\big)w$$
 набуває вигляду
\begin{equation}\label{f3-6-8}
 E\dot{x}=A_0x+B_0w,\quad z=C_0x+D_0w,
\end{equation}
де \;$K_0=(I_m-K\bar{D}_{22})^{-1}K$,\; $\det (I_m-K\bar{D}_{22})\ne 0$,
$$A_0=A+B_2\big(K_{10}C_2+K_{11}K_0\bar{C}_2E_2^\top\big),\quad B_0=B_1+B_2\big(K_{10}D_{21}+K_{11}K_0\bar{D}_{21}\big),$$
$$C_0=C_1+D_{12}\big(K_{10}C_2+K_{11}K_0\bar{C}_2E_2^\top\big),\quad D_0=D_{11}+D_{12}\big(K_{10}D_{21}+K_{11}K_0\bar{D}_{21}\big).$$

\medskip
\textbf{3.1. Динамічний регулятор.} Застосовуючи для звичайної системи (\ref{f3-2}) динамічний регулятор
\begin{equation}\label{f3-8}
  \dot{\eta}=Z\eta+Vy,\quad u=U\eta+Ky,\quad \eta(0)=0,
\end{equation}
замкнена система у розширеному фазовому просторі $\mathbb{R}^{r+p}$ має вигляд
\begin{equation}\label{f3-8-1}
  \dot{\widehat{x}}=\widehat{A}_*\widehat{x}+\widehat{B}_*w,\quad z=\widehat{C}_*\widehat{x}+\widehat{D}_*w,\quad \widehat{x}(0)=\widehat{x}_0,
\end{equation}
де
$$
 \widehat{A}_*=\widehat{A}+\widehat{B}_2\widehat{K}_0\widehat{C}_2,\quad
 \widehat{B}_*=\widehat{B}_1+\widehat{B}_2\widehat{K}_0\widehat{D}_{21},
 $$$$
 \widehat{C}_*=\widehat{C}_1+\widehat{D}_{12}\widehat{K}_0\widehat{C}_2,\quad
 \widehat{D}_*=\widehat{D}_{11}+\widehat{D}_{12}\widehat{K}_0\widehat{D}_{21},
$$
$$
\widehat{x}=\left[\begin{array}{c}
      \xi_1 \\
      \eta \\
    \end{array}\right],\quad
\widehat{A}=\left[\!\begin{array}{cc}
                  \bar{A} & 0_{r\times p} \\
                  0_{p\times r} & 0_{p\times p} \\
                \end{array}\!\!\right],\quad
\widehat{B}_1=\left[\begin{array}{c}
               \bar{B}_1 \\
              0_{p\times s} \\
             \end{array}\right],\quad
                \widehat{B}_2=\left[\!\begin{array}{cc}
                  \bar{B}_2 & 0_{r\times p} \\
                  0_{p\times m} & I_p \\
                \end{array}\!\!\right],
$$
$$\widehat{C}_1=\left[\begin{array}{cc}
                    \bar{C}_1 & 0_{k\times p} \\
                  \end{array}\right],\quad \widehat{D}_{11}=\bar{D}_{11},\quad
\widehat{D}_{12}=\left[\begin{array}{cc}
                     \bar{D}_{12} & 0_{k\times p} \\
                   \end{array}\right],$$
$$\widehat{C}_2=\left[\!\begin{array}{cc}
                    \bar{C}_2 & 0_{l\times p} \\
                   0_{p\times r} & I_p  \\
                \end{array}\!\!\right],\quad \widehat{D}_{21}=\left[\begin{array}{c}
                      \bar{D}_{21} \\
                     0_{p\times s} \\
                     \end{array}\right],\quad \widehat{K}_0=\left[\begin{array}{cc}
                  K_0 & U_0 \\
                  V_0 & Z_0 \\
                \end{array}\right],$$
$$K_0=(I_m-K\bar{D}_{22})^{-1}K,\quad U_0=(I_m-K\bar{D}_{22})^{-1}U, $$
$$V_0=V(I_l-\bar{D}_{22}K)^{-1},\quad Z_0=Z+V\bar{D}_{22}(I_m-K\bar{D}_{22})^{-1}U.$$

Визначимо критерій якості $\widehat{J}$ системи (\ref{f3-8-1}) виду (\ref{f1-2}) з ваговими матрицями $P$, $Q$ і $\widehat{X}_0$, де $\widehat{X}_0$ --- деяка блокова матриця розміру $(r+p)\times (r+p)$, першим діагональним блоком якої є $\bar{H}$.
Оскільки початковий вектор регулятора (\ref{f3-8}) нульовий, то значення $\widehat{J}$ збігається з $J$.

 \begin{lemma} {\rm \cite{Mazko-NDST-2017}}. \label{l3-1}
Для заданих додатно визначених матриць $X,Y\in R^{r\times r}$ і числа $\gamma>0$ існують матриці $X_1\in \mathbb{R}^{p\times r}$, $X_2\in \mathbb{ R}^{p\times p}$, $Y_1\in \mathbb{R}^{p\times r}$ і $Y_2\in \mathbb{R}^{p\times p}$, що задовольняють співвідношення
\begin{equation}\label{f3-7-1}
\widehat{X}=\left[\!\begin{array}{cc}
                  X & \!X_1^\top   \\
                  X_1 & \!X_2 \\
                  \end{array}
                  \!\right]>0,\quad  \widehat{Y}=\left[\!\begin{array}{cc}
                  Y & \!Y_1^\top   \\
                  Y_1 & \!Y_2 \\
                  \end{array}
                  \!\right]>0,\quad \widehat{X}\widehat{Y}=\gamma^2I_{r+p},
\end{equation}
тоді і лише тоді, коли
 \begin{equation}\label{f3-9}
   W=\left[\begin{array}{cc}
        X & \gamma I_r \\
        \gamma I_r & Y
      \end{array}
\right]\ge 0,\quad {\rm rank}\,W\le r+p.
 \end{equation}
  \end{lemma}

 \begin{theorem}\label{t3-3}
Для системи {\rm (\ref{f3-1})} існує динамічний регулятор {\rm(\ref{f3-8})} порядку $p\le r$,
при якому замкнена система є допустимою і її критерій якості $J<\gamma$, тоді і лише тоді, коли сумісна щодо $X$ і $Y$ система співвідношень
{\rm(\ref{f2-6})}, {\rm(\ref{f3-4})}, {\rm(\ref{f3-5})} і {\rm(\ref{f3-9})}.
Матриці такого регулятора можна визначити у вигляді
\begin{equation}\label{f3-10}
\left[\begin{array}{cc}
                  K & U \\
                  V & Z \\
                \end{array}\right]=(I_{m+p}+\widehat{K}_0\widehat{D}_{22})^{-1}\widehat{K}_0,\quad
\widehat{D}_{22}=\left[\!\begin{array}{cc}
                  \bar{D}_{22} & 0_{l\times p} \\
                  0_{p\times m} & 0_{p\times p} \\
          \end{array}\!\right],
\end{equation}
де $\widehat{K}_0$ --- розв'язок ЛМН
\begin{equation}\label{f3-11}
 \widehat{L}^\top \widehat{K}_0\widehat{R}+\widehat{R}^\top \widehat{K}_0^\top \widehat{L}+\widehat{\Omega}<0,
\end{equation}
$$
\widehat{R}=\left[\begin{array}{ccc}
                \widehat{C}_2 & \widehat{D}_{21} & 0_{l+p\times k} \\
              \end{array}\right],\quad
\widehat{L}=\left[\begin{array}{ccc}
                \widehat{B}^\top_2\widehat{X} & 0_{m+p\times s} & \widehat{D}^\top _{12} \\
              \end{array}\right],$$
$$\widehat{\Omega}=\left[\begin{array}{ccc}
         \widehat{A}^\top \widehat{X}+\widehat{X}\widehat{A} & \widehat{X}\widehat{B}_1 & \widehat{C}^\top _1 \\
         \widehat{B}^\top _1\widehat{X} & -\gamma^2P& \widehat{D}^\top _{11} \\
         \widehat{C}_1 & \widehat{D}_{11} & -Q^{-1} \\
         \end{array}\right].$$
Блокова матриця $\widehat{X}$ в {\rm(\ref{f3-11})} сформована на основі леми {\rm \ref{l3-1}} згідно з {\rm(\ref{f3-7-1})}, де $X$ і $Y$ задовольняють співвідношення
{\rm(\ref{f2-6})}, {\rm(\ref{f3-4})}, {\rm(\ref{f3-5})} і {\rm(\ref{f3-9})}.
\end{theorem}

Враховуючи структуру матриць в (\ref{f3-8-1}), систему (\ref{f3-2}) з динамічним регулятором (\ref{f3-8}) можна представити у вигляді системи у просторі $\mathbb{R}^{r+p}$ зі статичним регулятором:
$$
  \dot{\widehat{x}}=\widehat{A}\,\widehat{x}+\widehat{B}_1\,w+\widehat{B}_2\,\widehat{u}, \quad
    z=\widehat{C}_1\,\widehat{x}+\widehat{D}_{11}\,w+\widehat{D}_{12}\,\widehat{u}, \quad
    \widehat{y}=\widehat{C}_2\,\widehat{x}+\widehat{D}_{21}\,w,
$$
$$\widehat{x}=\left[\begin{array}{c}
      \xi_1 \\
      \eta \\
    \end{array}\right],\quad
        \widehat{y}=\left[\begin{array}{c}
      y-\bar{D}_{22}u \\
      \eta \\
    \end{array}\right],\quad \widehat{u}=\left[\begin{array}{c}
      u \\
      \dot{\eta} \\
    \end{array}\right],\quad \widehat{u}=\widehat{K}_0\widehat{y}.
$$
Тому теорему \ref{t3-3} можна встановити як наслідок теореми \ref{t3-1} і леми \ref{l3-1}.

Зазначимо, що теореми \ref{t3-1} і \ref{t3-3} без використання обмеження $X<\gamma^2\bar{H}$ дають критерії існування та методи побудови стабілізуючих регуляторів, що забезпечують оцінку $J_0<\gamma$ для відповідних замкнених систем. Умови теореми \ref{t3-3} у випадку $p = 0$ є критерієм існування статичного регулятора (\ref{f3-3}) із вказаними властивостями в теоремі \ref{t3-1}. Побудова динамічних регуляторів порядку
$p = r $, що задовольняють теорему \ref{t3-3}, зводиться до розв'язання системи ЛМН без додаткових обмежень. У цьому випадку рангове обмеження в (\ref{f3-9}) виконується автоматично.

\medskip
Наведемо алгоритм побудови динамічного регулятора (\ref{f3-8}), що задовольняє теорему \ref{t3-3}.

\begin{algorithm} \label{a3-1}\rm

$\;$\\
\noindent {\rm 1)} Обчислення матриць перетворення {\rm (\ref{f2-2})} і системи {\rm (\ref{f3-2})};

\noindent {\rm 2)} обчислення матриць $W_{\bar{R}}$ і $W_{\bar{L}}$, де $\bar{R}=\left[\begin{array}{cc}
                \bar{C}_2 & \bar{D}_{21} \\
              \end{array}\right]$, $\bar{L}=\left[\begin{array}{cc}
                \bar{B}_2^\top & \bar{D}_{12}^\top \\
              \end{array}\right]$;

\noindent {\rm 3)} знаходження матриць $X=X^\top >0$ і $Y=Y^\top >0$, які задовольняють співвідношення {\rm(\ref{f2-6})}, {\rm(\ref{f3-4})}, {\rm(\ref{f3-5})} і {\rm(\ref{f3-9})};

\noindent {\rm 4)} побудова розкладу $\Delta=Y-\gamma^2X^{-1}=S^\top S\ge 0$, де $S\in {\mathbb R}^{p\times r}$, ${\rm ker}\,S={\rm ker}\,\Delta$, та формування блокової матриці
 $$\widehat{X}=\left[\begin{array}{cc}
                  X&\!X_1^\top\\
                  X_1&\!X_2\\
                  \end{array} \right]>0,\quad X_1=\dfrac{1}{\gamma}SX,\quad X_2=\dfrac{1}{\gamma^2}SXS^\top +I_p;$$

\noindent {\rm 5)} розв'язання лінійної матричної нерівності {\rm (\ref{f3-11})} щодо $\widehat{K}_0$ за умови \\$\det (I_m+K_0\bar{D}_{22})\ne 0$;

\noindent {\rm 6)} обчислення матриць регулятора за формулою {\rm (\ref{f3-10})}.
\end{algorithm}

\medskip
Даний алгоритм може бути реалізований, наприклад, засобами системи MATLAB. Якщо в п. 4 алгоритму $\Delta=0$, тобто ${\rm rank}\,W=r$, то, розв'язуючи ЛМН (\ref{f3-7}), отримаємо статичний регулятор (\ref{f3-3}), що задовольняє теорему \ref{t3-1}.

\begin{remark}\label{r2-2}
Якщо пара матриць \{$E,A$\} у системі (\ref{f3-3}) є імпульною, то керування для системи (\ref{f3-1}) можна шукати у вигляді $u=K_1y+v$, де $K_1$ --- матриця допоміжного статичного регулятора за умов (\ref{f3-6-6}), а $v$ --- нове керування, яке створює динамічний регулятор
$$ \dot{\eta}=Z\eta+Vy,\quad v=U\eta+Ky,\quad \eta(0)=0,$$
і яке розв'язує поставлену задачу для звичайної системи (\ref{f3-2-1}), побудованої на основі еквівалентного перетворення системи (\ref{f3-6-7}) (див. попередній підрозділ). В результаті замкнена дескрипторна система у розширеному фазовому просторі має вигляд
\begin{equation}\label{f3-12}
 \widehat{E}\dot{\widehat{x}}=\widehat{A}_0\widehat{x}+\widehat{B}_0w,\quad z=\widehat{C}_0\widehat{x}+\widehat{D}_0w,\quad \widehat{x}(0)=\widehat{x}_0,
\end{equation}
де
$$\widehat{E}=\left[\begin{array}{cc}
                         E & 0 \\
                        0 & I_p \\
                  \end{array}\right],\quad
\widehat{x}=\left[\begin{array}{c}
      x \\
      \eta \\
    \end{array}\right],\quad
\widehat{x}=\left[\begin{array}{c}
      x_0 \\
      0 \\
    \end{array}\right],
    $$
$$\widehat{A}_0=\!\left[\!\begin{array}{cc}
                     A+B_2(K_{10}C_2\!+\!K_{11}G_1) & B_2K_{11}G_2 \\
                    V(\bar{C}_2E_2^\top\!+\!\bar{D}_{22}G_1) & \!\!\!Z+V\bar{D}_{22}G_2 \\
                   \end{array}\!\right],
 \widehat{B}_0=\!\left[\!\begin{array}{c}
                   B_1+B_2(K_{10}D_{21}\!+\!K_{11}G_3) \\
                   V(\bar{D}_{21}+\bar{D}_{22}G_3) \\
                 \end{array}\!\right],
$$
$$\widehat{C}_0=\left[\!\begin{array}{cc}
                    C_1+D_{12}(K_{10}C_2\!+\!K_{11}G_1) & D_{12}K_{11}G_2 \\
                  \end{array}\!\right],\; \widehat{D}_0=D_{11}+D_{12}(K_{10}D_{21}+K_{11}G_3),
$$
$$G_1=K_0\bar{C}_2E_2^\top,\;\; G_2=(I_m-K\bar{D}_{22})^{-1}U,\;\;  G_3=K_0\bar{D}_{21},\;\;  K_0=(I_m-K\bar{D}_{22})^{-1}K.$$
\end{remark}

\medskip

\textbf{4. Приклад.}
Розглянемо лінеаризовану модель гідравлічної системи з трьома послідовно сполученими резервуарами, яка описується у вигляді дескрипторної системи керування (\ref{f3-1}) з такими матрицями \cite{Araujo-2012}:
$$
E=\left[\!\begin{array}{ccc}
      1 & 0 & 0 \\
      0 & 1 & 0 \\
      0 & 0 & 0 \\
    \end{array}\!\right],\,
A=\left[\!\begin{array}{ccc}
      \!-k_1 & 0 & 0 \\
      k_1 & \!\!-k_2 & 0 \\
      1 & 1 & 1 \\
    \end{array}\!\right],\,
B_1=\left[\!\begin{array}{cc}
       1 & 0 \\
       0 & 0 \\
       0 & 0 \\
     \end{array}\!\right],\,
 B_2=\left[\!\begin{array}{c}
       1 \\
       0 \\
       0 \\
     \end{array}\!\right],
$$
$$
C_1=\left[\begin{array}{ccc}
       0 & 0 & 1 \\
     \end{array}\right],\quad
D_{11}=0_{1\times 2},\quad D_{12}=1,
$$
$$
C_2=\left[\begin{array}{ccc}
       1 & 0 & 0 \\
       0 & 1 & 0 \\
     \end{array}\right],\quad
D_{21}=\left[\begin{array}{cc}
         0 & 0 \\
         0 & 1 \\
     \end{array}\right],
         \quad D_{22}=0_{2\times 1}.
$$

\begin{figure}[h]
  \centering
  \includegraphics[scale=0.3]{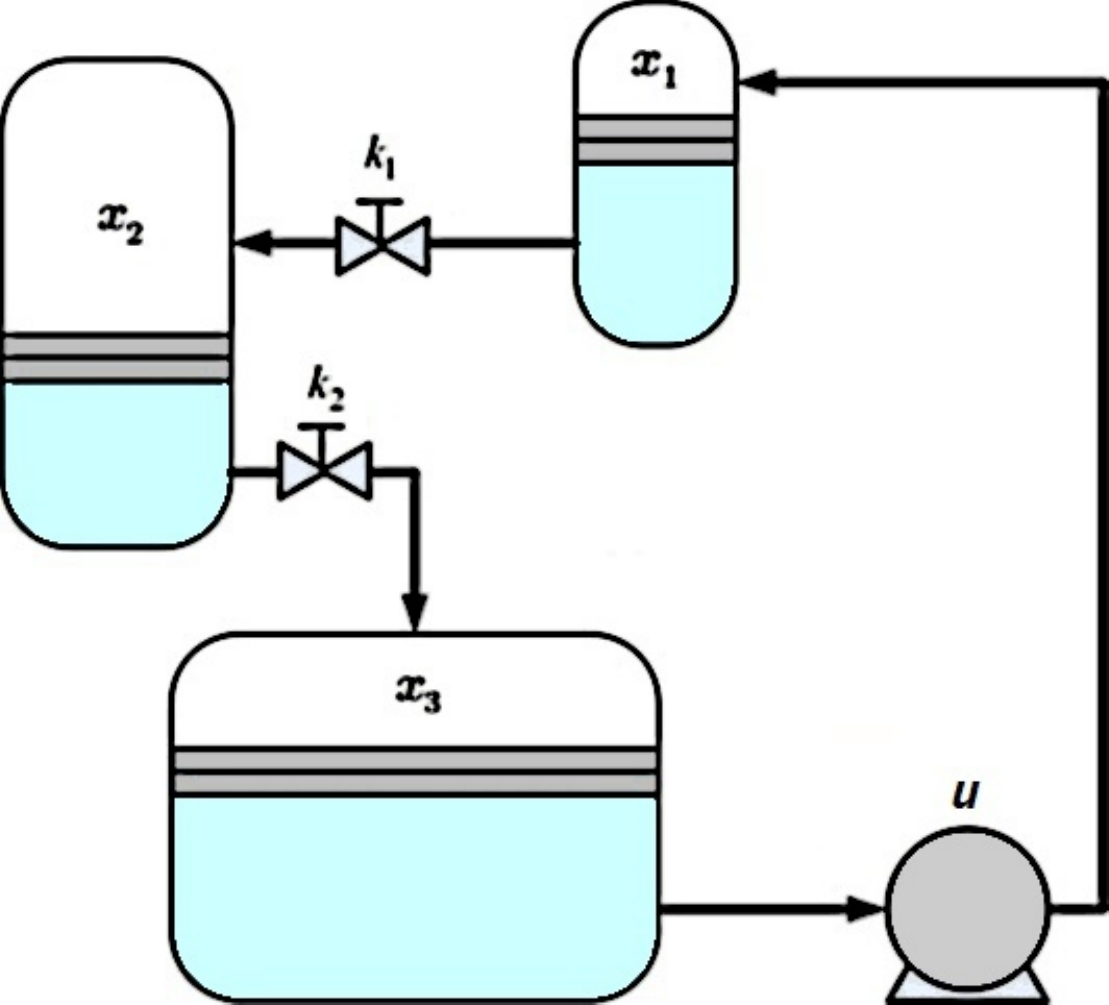}
  \caption{Гiдравлiчна система з трьома резервуарами.}\label{fig-1}
\end{figure}

Компоненти вектора стану $x=\left[\begin{array}{ccc}
                                x_1 & x_2 & x_3 \\
                              \end{array}\right]^\top$
визначають рівні рідини у відповідних резервуарах, вектор $w=\left[\begin{array}{cc}
                                w_1 & w_2 \\
                              \end{array}\right]^\top$
формують неконтрольоване збурення $w_1$ й похибка $w_2$ у вимірах $y=\left[\begin{array}{cc}
                                x_1 & x_2+w_2 \\
                              \end{array}\right]^\top$,
керований вихід $z=x_3+u$, а роль керування $u$, що регулює рівні рідини у перших двох резервуарах, виконує дебіт (потік) рідини в насосі із третього резервуара в перший (рис. \ref{fig-1}).

Виберемо вагові матриці критерію якості (\ref{f1-2}) і допустимі значення параметрів
$$P={\rm diag}\{2,1\},\; Q=1,\; X_0={\rm diag}\{3,2,0\},\; H={\rm diag}\{3,2,1\},\; k_1=1,\; k_2=1,5.$$

У даному прикладі $n=3$, $m=1$, $k=1$, $s=2$, $l=2$, $r=2$, пара матриць ($E,A$) допустима, а трійки матриць ($E,A,B_2$) і ($E,A,C_2$)  --- відповідно $I$-керована та $I$-спостережувана. Система без керування має критерій якості $J=1,17851$.

Для системи (\ref{f3-1}) на основі теореми \ref{t3-1} при $\gamma=1$ знайдено матрицю статичного регулятора (\ref{f3-3})
$$K=\left[\begin{array}{cc}
           -0,67954	&   -0,39810\\   	
    \end{array}\right],$$
при якому замкнена система є допустимою і її критерій якості $J=0,97989<\gamma$. При цьому її скінченний спектр збігається з
$\sigma(A_*)=\big\{-1,58977\pm  0,62453\,i\big\}$, де $A_*$ --- матриця системи (\ref{f3-3-1}).

 \begin{figure}[h!]
 \begin{minipage}[h]{62mm}
 \begin{center}
 \vspace*{-5mm}
  \includegraphics[scale=0.46]{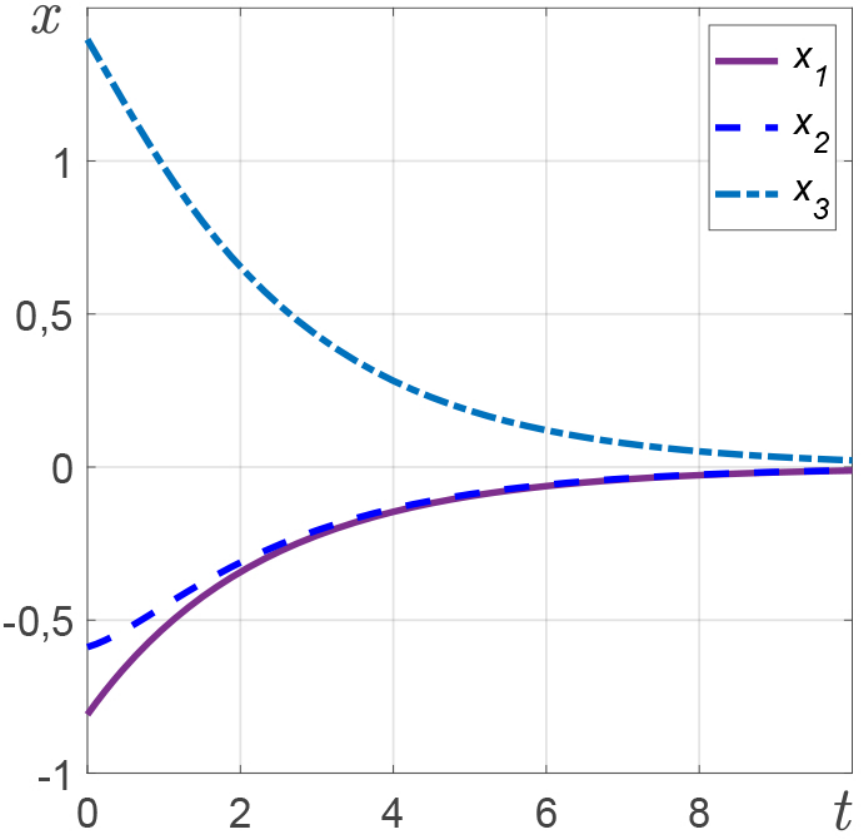}
 \caption{\small  Поведінка замкненої системи.}
 \label{fig-2}
 \end{center}
 \end{minipage}
 \hspace*{15mm}
 \begin{minipage}[h]{62mm}
 \begin{center}
 \includegraphics[scale=0.46]{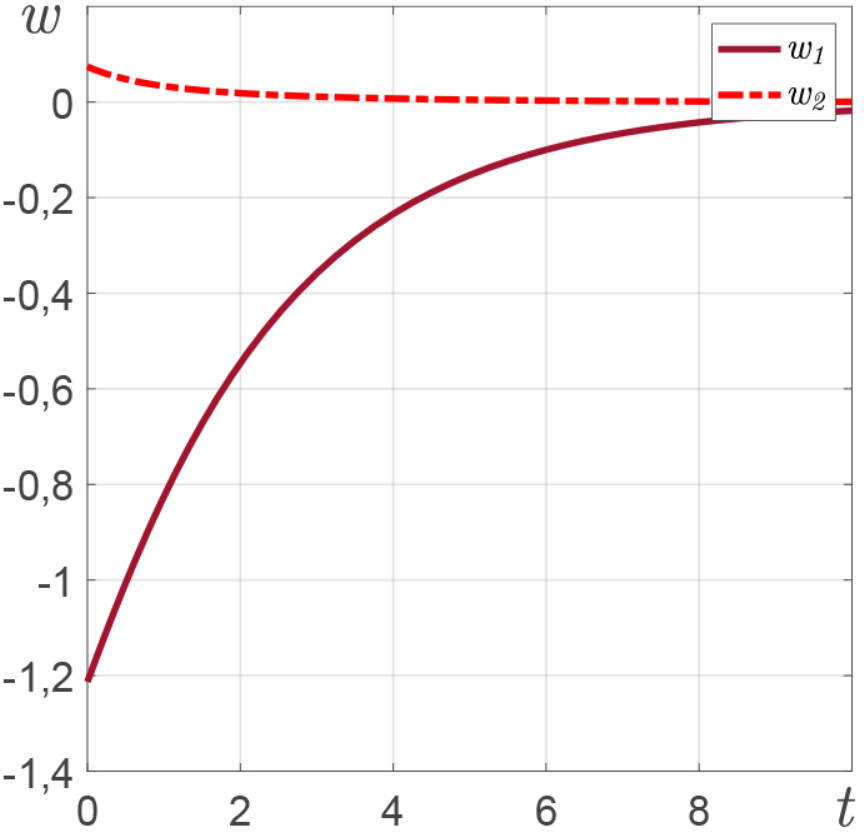}
\caption{\small Найгірше збурення щодо \\критерію якості $J$.}
 \label{fig-3}
 \end{center}
 \end{minipage}
 \end{figure}

Застосовуючи лему \ref{l2-2} для замкненої системи, знайдено найгірше збурення (\ref{f2-11}) і найгірший початковий вектор (\ref{f2-8}) стосовно критерію якості $J$:
\begin{equation}\label{f4-1}
w=\bar{K}_*\xi_1,\quad \bar{K}_*=\left[\begin{array}{cc}
    1,12785	 &   0,51161\\	
   -0,24149	 &   0,20639\\	
      \end{array}\right],
\end{equation}
\begin{equation}\label{f4-2}
x_0=\left[\begin{array}{ccc}
                     -0,80868 & -0,58825 & 1,39693 \\
                   \end{array}\right]^\top.
\end{equation}
На рис. \ref{fig-2} зображена поведінка розв'язку замкненої системи за найгірших умов (\ref{f4-1}) і (\ref{f4-2}),
 а на рис. \ref{fig-3} ---  вектор-функція (\ref{f4-1}) найгіршого збурення.

 Далі, застосовуючи алгоритм \ref{a3-1}, для системи (\ref{f3-1}) знайдено матриці наближеного $J$-опти\-мального динамічного регулятора (\ref{f3-8}) порядку $p=r$
 $$
 Z=\left[\begin{array}{cc}
       -0,62278 & 0,00590 \\
       -0,00004 & -0,92267 \\
     \end{array}\right],\quad
 V=\left[\begin{array}{cc}
       0,00047 & 0,00025 \\
       -0,00126 & 0,00113 \\
     \end{array}\right],
 $$
 $$
 U=\left[\begin{array}{cc}
       -0,00152 & -0,00021 \\
     \end{array}\right],\quad
 K=\left[\begin{array}{cc}
      -1,09088 & -0,04544 \\
    \end{array}\right],
 $$
при якому замкнена система є допустимою, має скінченний спектр
$$\big\{-1,99999,\; -1,59089,\; -0,92267,\; -0,62278\big\}$$
і мінімальне значення критерію якості $J=0,97895$.

\medskip
\textbf{5. Висновки.} Розроблено конструктивні методи оцінки та досягнення бажаного рівня гасіння зовнішніх і початкових збурень у дескрипторних системах керування за допомогою статичних та динамічних регуляторів. Практична реалізація даних методів базується на еквівалентному перетворенні дескрипторних систем і застосуванні відомих методів теорії $H_\infty$-керування для звичайних систем меншого порядку. Так, умови існування і алгоритми побудови динамічного регулятора порядку $p={\rm rank}\,E$, при якому замкнена система є допустимою зі зваженими критеріями якості $J_0<\gamma$ або $J<\gamma$, зводяться до розв'язування систем ЛМН без додаткових рангових обмежень. Якщо вихідна дескрипторна система не є неімпульсною, то запропоновано пошук додаткового керування, яке забезпечує вказану властивість даної системи. Еквівалентне перетворення дескрипторної системи до звичайної застосовано також при знаходженні векторів найгірших зовнішніх і початкових збурень щодо зважених критеріїв якості. Вивчення поведінки замкненої системи за таких найгірших умов може бути важливим при конструюванні і випробуванні реальних керованих об'єктів.

\medskip

{\small

\renewcommand\refname{\normalsize Література}

}

\newpage

\UDC{УДК 517.925.51;\,681.5.03}

\Title{Еквівалентне перетворення і зважена $H_\infty$-оптимізація лінійних дескрипторних систем}


\Author{\large О.\,Г.~Мазко}

\Address{Інститут математики НАН України, Київ}

\vspace{1mm}

\Abstract{Досліджується проблема узагальненого $H_\infty$-керування для класу допустимих дескрипторних систем з ненульовим початковим вектором.
Використовується узагальнений критерій якості, що характеризує зважений рівень гасіння зовнішніх і початкових збурень.
Пропонується невироджене перетворення системи, яке дозволяє застосовувати відомі методи оцінки та досягнення бажаних критеріїв якості для звичайних систем меншого порядку. Наводиться числовий приклад керованої гідравлічної системи з трьома резервуарами.}

%
%
%
%
%

\vspace{5mm}

\Title{Equivalent transformation and weighted $H_\infty$-optimization of linear descriptor systems}

\Author{\large A.\,G.~Mazko}

\Address{Institute of Mathematics, Ukrainian Academy of Sciences,
Kyiv}

\vspace{1mm}

\Abstract{The problem of a generalized type of $H_\infty$-control is investigated for a class of admissible descriptor systems with a non-zero initial vector. A generalized performance measure is used, which characterizes the weighted damping level of external and initial disturbances.
A non-degenerate transformation of the system is proposed, which allows to apply known methods for evaluation and the achievement of desired performance measures for conventional lower-order systems. A numerical example of a controlled hydraulic system with three tanks is given.}

{\small Mathematics Subject Classification (2010):
93C35, 
34D20, 
37N35. 

\newpage

\begin{center}
{\Large\bf Відомості про автора.}
\end{center}

\textbf{\emph{МАЗКО Олексій Григорович}}, завідувач відділу Інституту
математики НАН України, член-кореспондент НАН України, доктор физ.- мат. наук, професор.

Адреса для листування:  Україна, 01024 Київ-4, вул. Терещенківська, 3,
Інститут математики НАН України.

E-mail: mazkoag@gmail.com,  mazko@imath.kiev.ua.

С. т.: (044)234-51-50; моб. (093)238-55-70.

\end{document}